\RequirePackage{amsmath}  
\documentclass[runningheads]{llncs}
\usepackage[T1]{fontenc}
\usepackage{cite}
\usepackage{graphicx}%
\usepackage{multirow}%

\usepackage{amsmath,amssymb,amsfonts}%
\usepackage{amsthm}%
\usepackage{bbm}
\usepackage{mathrsfs}%
\usepackage[title]{appendix}%
\usepackage{xcolor}%
\usepackage{textcomp}%
\usepackage{manyfoot}%
\usepackage{booktabs}%
\usepackage{algorithm}%
\usepackage{algorithmic}%

\usepackage{listings}%
\usepackage{hyperref}       
\usepackage[capitalise]{cleveref}

\newcommand{\sm}[2]{\begin{smallmatrix}\item1\\\item2 \end{smallmatrix}}

\newcommand\ddfrac[2]{\frac{\displaystyle \item1}{\displaystyle \item2}}

\newcommand{\mc}[1]{\mathbb{\item1}}

\usepackage{mathtools}

\usepackage{tikz}
\usepackage{ifthen}
\usetikzlibrary{arrows,automata,positioning,shapes,calc, decorations.pathmorphing}

\renewcommand{\le}{\leqslant}

\newcommand{\cE}{{\mathcal{E}}}

\newcommand{\cG}{{\mathcal{G}}}

\newcommand{\cT}{{\mathcal{T}}}

\newcommand{\cV}{{\mathcal{V}}}

\newcommand{\cbr}[1]{\left( #1 \right)}

\newcommand{\br}[1]{\left\{ #1 \right\}}
\def\<#1,#2>{\langle #1,#2\rangle}

\usepackage{todonotes}

\usepackage{color}

\usepackage{xspace}

\begin{document}
\title{
An Equilibrium Dynamic Traffic Assignment Model with Linear Programming Formulation
}

\titlerunning{An Equilibrium Dynamic Traffic Assignment Model}

\author{
  Victoria Guseva\inst{1},
  Ilya Sklonin\inst{1},
  Irina Podlipnova\inst{1},
  Demyan Yarmoshik\inst{1,3},
  Alexander Gasnikov\inst{2,3,4}
}
  
  \authorrunning{V. Guseva et al.}
\institute{Moscow Institute of Physics and Technology, Dolgoprudny, Russia
  \\\email{vigus1810@gmail.com}, \email{\{sklonin.ilya, podlipnova.iv, yarmoshik.dv\}@phystech.edu}
  \and
  Innopolis University, Innopolis, Russia
  \\\email{gasnikov@yandex.ru}
  \and
  Institute for Information Transmission Problems, Moscow, Russia
  \and
  Caucasus Mathematical Center, Adyghe State University, Maikop, Russia
}

\maketitle              

\begin{abstract}
In this paper, we consider a~dynamic equilibrium transportation problem. There is a~fixed number of cars moving from origin to destination areas. Preferences for arrival times are expressed as a~cost of arriving before or after the preferred time at the destination. Each driver aims to minimize the time spent during the trip, making the time spent a~measure of cost. The chosen routes and departure times impact the network loading. The goal is to find an equilibrium distribution across departure times and routes. 

For a relatively simplified transportation model we show that an equilibrium traffic distribution can be found as a solution to a linear program.
In earlier works linear programming formulations were only obtained for social optimum dynamic traffic assignment problems.
We also discuss algorithmic approaches for solving the equilibrium problem using time-expanded networks.
\end{abstract}

\section{Introduction}
The equilibrium traffic assignment problem is fundamental to transportation modelling. 
Essentially, the problem is to find the distribution of flows in a transportation network generated by the selfish route choices of drivers. 
Solving for such a distribution with various network parameters allows predicting network load and facilitates informed decision-making for better network design.
There are two different approaches to traffic assignment: static and dynamic.

In this paper we develop a simple dynamic equilibrium model that reduces to a minimum cost multicommodity flow problem, similar to the static Stable Dynamics model \cite{nesterov2003stationary}. 
It benefits from lower computational requirements compared to other dynamic models while avoiding the drawbacks of static models, such as relying on the contradictory notion of link performance functions in the Beckmann model~\cite{beckmann1956studies} or the infeasibility of the Stable Dynamics model in peak hours.

\subsection{Literature Review}
\textbf{Static traffic assignment.} Static models assume that network flows do not depend on the time variable and thus can only capture the average characteristics of the network's loading for a given period of time.
The Beckmann model~\cite{beckmann1956studies} with Bureau of Public Roads (BPR) functions representing the dependency between flow density and travel time on a road segment is probably the most popular model~\cite{boyles2023transportation,reeder2012positive}.
Paper \cite{nesterov2003stationary} presented the alternative Stable Dynamics static model, where the link flows are strictly limited by the capacities, motivated by the logical inconsistencies of the Beckmann model.

A significant advantage of static models is their computational efficiency. 
Static equilibrium models can be formulated as convex optimization problems \cite{beckmann1956studies} or variational inequalities if flow interactions are considered  \cite{dafermos1980traffic,smith1979existence}.
Many highly performant algorithms were designed specifically for such problems \cite{mitradjieva2013stiff,babazadeh2020reduced,kubentayeva2021finding,ignashin2024modifications}.
This allows for the development of combined travel demand models based on static models, which can still be solved within a reasonable time limit for decent accuracy in large networks \cite{florian2002multi,kubentayeva2024primal}.

\textbf{Dynamic traffic assignment.} First dynamic traffic assignment models were suggested in \cite{merchant1978model,merchant1978optimality}.
However, they were limited to searching only system optimum, which means that each single driver wants to reduce not his personal cost, but total expenses of all drivers in a system, according to Wardrop's second principle. These first models were limited to single destination and single commodity due to computational complexity of dynamic models.
In the work of Carey \cite{carey1992nonconvexity} it was shown that when models consider multicommodities and multiple destinations then we should also think about first-in-first-out (FIFO) condition. Independently of what type of equilibrium we search -- user or system -- there is non-convexity arise due to FIFO constraint. 
Author of the paper considers FIFO condition on average. This means that in reality some particular vehicles can pass each other and travel at different speed, but on average traffic which enter a road first will exit from it first.

Similar idea is highlighted in 
\cite{zhang2005some}. The authors conclude that FIFO is not a necessary constraint for building dynamic traffic assignment problem, however FIFO assumption can simplify building link and traffic dynamic models.

The study and analysis of traffic congestion during rush hours have a long-standing history, beginning with the pioneering work of Vickrey in 1969 \cite{vickrey1969congestion}. 
In Vickrey's fundamental model, it is assumed that a set of commuters, aim to reach a single destination using a single route that features a bottleneck with limited capacity at the desired time, which is evenly distributed over a closed interval. Each commuter decides on their departure time from home to minimize their total travel cost, which includes travel time, waiting time at the bottleneck, and the costs of arriving early or late at their destination. 
Several authors have extended this model. 
For example, \cite{smith1984existence} investigated the existence of  equilibrium for a wider class of arrival cost functions. The follow-up work \cite{daganzo1985uniqueness} described additional properties of the equilibrium in the model of \cite{smith1984existence}. 
Papers \cite{arnott1992route,liu2015semi} consider heterogeneities in the cost of travel time or delay in the schedule. There are also works (for example \cite{daganzo1985uniqueness}) in which different types of user heterogeneities are assumed (work start times and value of times), but with a very stringent assumption about the piecewise linearity of the schedule function. However, in \cite{nie2009numerical}, the authors showed that there is no guaranteed convergence in the straightforward formulation of the equilibrium conditions. As noted earlier, all these works consider the simplest network design~---  the case of a single origin-destination pair connected by one route with a single bottleneck. In the articles \cite{li2020fifty} and \cite{akamatsu2021new} the authors decided to systematize the models of departure-time choice equilibrium (DTCE) for this type of road network; in the second work the authors additionally proposed a novel approach to modeling departure time equilibrium. They utilize continuous-time linear programming and optimal transport theory to achieve analytical solutions. Special attention is given to the Monge property, which determines sorting patterns of equilibrium, allowing models to account for multiple types of user heterogeneity.

There are also many studies focusing on other types of road networks. The paper by \cite{akamatsu2015corridor} considered a road system with several bottlenecks, while in \cite{iryo2007equivalent}, the authors examined a network with multiple origin-destination pairs and several roads between them, each containing a single bottleneck.

\medskip
\subsection{Paper Structure}
In \cref{sec:problem_statement} we formulate an equilibrium departure time and route choice problem in continuous time, and provide a toy example which can be solved analytically.
Then, in \cref{sec:discrete} we derive a formulation of a discrete variant of the problem as a~minimum cost multicommodity flow problem.
We show that the solution to the obtained optimization problem is a user equilibrium network loading.
Last, in \cref{sec:algo} we discuss numeric algorithms suitable for solving the discretized problem.


\section{Problem Statement in Continuous Time}\label{sec:problem_statement}

\subsection{Transportation Model}

\textbf{Transportation Network.} We consider the road network's representation as a simple directed graph, where nodes correspond to junctions and centroids (artificial origin and destination nodes) and links correspond to road segments and centroid connectors.

We would like to represent the entire transportation network as a simple directed graph, where links are associated with capacities $\hat x_e$ and free-flow times $\bar t_e$, without offloading any additional information, such as allowed turn directions, to other data structures.
This can be achieved by splitting each node corresponding to a junction in the following way.
For each input link $(v_1, v_2)$ adjacent to the intersection node $v_2$, we replace this node in the link with an artificial input node $v_{12}$.
Similarly, we replace the first node in an output link $(v_2, v_3)$ with a node $v_{23}$. 
The splitting procedure is illustrated at \cref{fig:cross_processing}. 
If the turn $(v_1, v_2, v_3)$ is allowed at the junction, we add a link $(v_{12}, v_{23})$ with capacity and free-flow time representing the turn capacity and the time required to make the turn. 
This results in that we have three disjoint subsets of links: $E = E_R \cup E_J \cup E_C$, where $E_R$ consists of links corresponding to road segments such as $(v_1, v_{12})$, $E_J$ consists of links corresponding to junctions such as $(v_{12}, v_{23})$ and $E_C$ are centroid connectors.
The resulting set of all nodes is denoted by $V$, and the whole transportation graph is denoted by $G = (V, E)$.
Sets of origin and destination centroids are denoted by $I \subset V$ and $J \subset V$  respectively.

\textbf{Travel Cost.} We assume that the travel cost, that each user strives to selfishly minimize, sums up from the travel time and a cost $\tau(t, T_0)$ which is a function of the actual arrival time $t$ parametrized by the desired arrival time $T_0$ (desired arrival times an be different for different users).
The arrival time cost represents the fact that most users aims to appear at some place at some specific time, and arriving late or too early is undesirable.
The travel time sums from free-flow travel times of traversed links $e \in E$ denoted by $\bar t_e$ and time spent in queues.  

This paper adopts the point queue link model to represent the queuing effect of congestion.
In the point queue model, queue hold no physical space: vehicles that cannot live the link just accumulate at the end of the link and do not interfere with the other vehicles, which travels freely to the end of the link.

We assume that all road segments have constant flow limit for all of its sections, and all capacity constraints are posed only for junctions.
Any change in link's capacity can be modeled by introducing an artificial junction, possibly with single input and single output.
We denote the capacity (in standardized vehicles per unit of time) of link $e \in E$ by $\hat x_{e}$, setting it to infinity for edges in $E_R$.



\textbf{Demands.}
Traffic between origins and destinations is represented by the time-dependent matrix $\{d_{ij}\}_{i \in I, j \in J}$, where $d_{ij}(t)$ is the total number/density of vehicles in zone $i$ that aim to arrive to zone $j$ at time $t$.

\begin{figure}[H]
\centering
\includegraphics[scale=0.54]{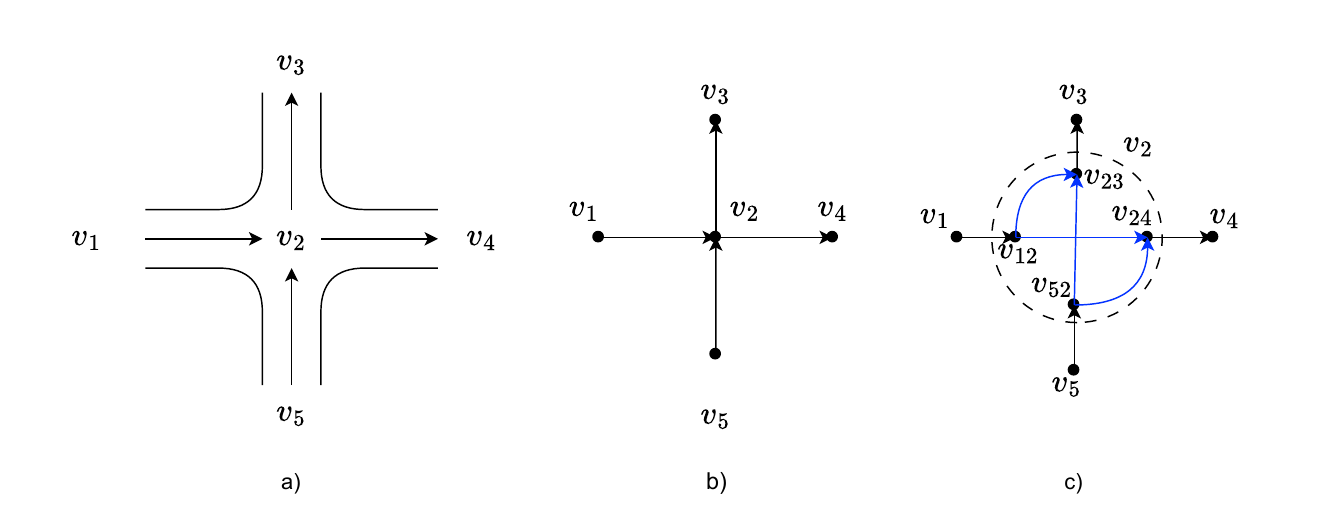}
\smallskip
\caption{Scheme of processing junctions. a) Schematic picture of a junction with indicated allowed directions of movement. b) Graph representation of the junction. c) Processed graph. Nodes inside the dashed circle are artificial nodes which correspond to node $v_2$. Blue links belong to subset $E_J$, while black links belong to subset $E_R$. }
\label{fig:cross_processing}
\end{figure}

\begin{figure}[!ht]
\centering
\includegraphics[scale=0.64]{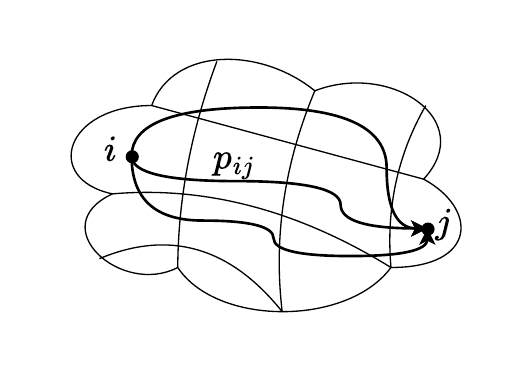}
\smallskip
\caption{Let consider some part of a city divided on zones. For example, zones $i \in O, j \in D$ are showed in the figure. There are several paths $p_{ij} \in P_{ij}$ between these nodes, where $ P_{ij}$ is a set of all possible paths with origin $i$ and destination $j$. }
\label{fig:zones}
\end{figure}

\subsection{Toy Example}
 To provide intuition about dynamic models in continuous time here we derive an equilibrium solution for a simple single origin-single destination network.
 This example was shared with us by Yurii Nesterov.
 A similar model differing in assumption about the distribution of arrival time preferences was also proposed in \cite{vickrey1969congestion}.










In the morning, \( N \) cars must leave a residential area to reach a work area. Each driver wants to arrive to the work area precisely at $T_0 =$ 9 a.m.. If a driver arrives after $T_0$, he may be late for work.
Conversely, arriving before $T_0$ results in lost sleep and waiting time before starting the workday.
We represent the costs of arriving later or earlier in the same units as travel costs.
Costs of arriving later or earlier are proportional to the time difference with coefficients $\alpha$ and $\beta$ respectively: $\tau(t, T_0) = \alpha \max\br{0, t - T_0}+ \beta \max\br{0, T_0 - t}$.

There is a single road between the residential and the work areas.
The usual travel time on a free road is $\bar t$ minutes. However, there is a narrow section midway along the road where the throughput is limited to $\hat x$ cars per hour.

We need to find the equilibrium distribution of cars depending on departure times from the residential area \( n(t) \).

It is impossible for all cars to arrive to the destination area at the same time. The driver arriving at exactly $T_0$ will be called a punctual driver. Let us denote his departure time as $t_2$.
Let also denote the time when the first car leaves the residential area by $t_1$.
The last car will depart at $t_3$.

If we want to thoroughly describe the car distribution we need to find the densities of our two flows: before and after the punctual driver \( n_1(t) \) and \( n_2(t) \), as well as the time borders limiting those distributions ($t_1$, $t_2$, $t_3$).


First, let's find the departure times of the first driver $t_1$ and last driver $t_3$. The first and the last driver don't need to to be stuck at the traffic jam, because they can move in the very beginning of the traffic flow when the jam had not formed yet or move in the end when the jam has already finished. According to our task the travellers costs depend on the time difference between the desired and actual arrival times to the destination area. The arrival time for the first driver can be expressed as ($t_1 + \bar t$), hence for the last driver it will be ($t_3 + \bar t$). The time difference between the desired and actual arrival time for the first driver ($T_0 - (t_1 + \bar t)$), and for the last driver ($(t_3 + \bar t) - T_0$). 

According to the first Wardrop principle, which is an application of the Nash's rule to the transportation task, the costs should be the same for all drivers or someone will change his strategy in order to optimise his costs. In our case the strategy means the departure time that the drivers choose to start their journeys. The equilibrium requires the same costs for the drivers.
 We can write an equation balancing the costs per minute of arriving early \( \beta \), and cost per minute of arriving late \( \alpha \)

 \begin{equation}
     \bar t + \beta(T_0 - (t_1 + \bar t)) = \bar t + \alpha ((t_3 + \bar t) - T_0).
 \end{equation}

 Obviously, the time from the first to the last car will be spread among all the other cars to pass through the bottleneck.

 Now lets calculate $t_3$. The term $\frac{N}{\hat x}$ denotes the total time taken by \( N \)  cars to pass through the bottleneck at a rate of $\hat x$ cars per hour
 
 \begin{equation}
     t_3 - t_1 = \frac{N}{\hat x}.
 \end{equation}

The previous two equations are resolved as a system to obtain $t_1$ and $t_3$:

\begin{equation}
t_1 = T_0 - \bar t - \frac{\alpha}{\alpha + \beta}  \frac{N}{\hat x}
\end{equation}

\begin{equation}
t_3 = T_0 - \bar t + \frac{\beta}{\alpha + \beta}  \frac{N}{\hat x}.
\end{equation}

As drivers travel, they either arrive late, early, or on time. Let's find the departure time  $t_2$ of the driver who arrived on time. We will call this driver punctual. The time for the punctual driver balancing out the costs adjusted for travel costs of the other drivers. As before, we will use the first Wardrop principle and consider that the benefit to each driver is the same

 \begin{equation}
     \bar t + \beta(T_0 - (t_1 + \bar t)) = T_0 - t_2.
 \end{equation}

Using previously found $t_1$ the solution for $t_2$ is:

\begin{equation}
t_2 = T_0 - \bar t - \frac{\alpha \beta}{(\alpha + \beta)}  \frac{N}{\hat x}.
\end{equation}

Now we can determine the distribution of cars that arrive before, after the punctual driver, and at the same time. All the cars that arrive before the punctual driver will form a car flow that value work time higher then rest time, on the other hand the cars that arrive after value their personal time higher than being at work at a certain time mark. This analysis focuses on the specific scenario when the road has a bottleneck with limited capacity.

The model is formulated based on the law of conservation of mass and Kirchhoff's first law, which asserts that traffic entering a road link must equal the traffic exiting, given no sources or sinks in between. We can find  \(n_1\) and \(n_2\) which are the traffic flow rates in two successive time intervals before and after the punctual driver and \(\hat x\) is the capacity of the bottleneck. The number of cars leaving the departure area, moving in the first part of the flow before the punctual driver can be found by multiplying the first traffic density \(n_1\) by the departure time difference between the first and the punctual driver. The number of cars of the same flow arriving to the destination area can be found by taking the multiplication of the bottleneck throughput by the time taken by the first part of the flow to come to the destination area from the first driver to the punctual driver. When we add the number of cars of our two flows we must get the total number of cars moving from the departure to the destination area. Lets resolve those equations as a system:

 \begin{equation}
 \begin{cases}
     n_1 (t_2 - t_1) = \hat x (T_0 - \bar t - t_1)\\

     n_2 (t_3 - t_2) + n_1 (t_2 - t_1) = N.\\
     \end{cases}\\
 \end{equation}

Lets find the time which takes to the first part of the flow to leave the departure area ($t_2 - t_1$)

\begin{equation}
{t_2 - t_1} = T_0 - \bar t - \frac{\alpha \beta}{(\alpha + \beta)}  \frac{N}{\hat x} - T_0 + \bar t + \frac{\alpha}{\alpha + \beta}  \frac{N}{\hat x} = \frac{\alpha }{(\alpha + \beta)}  \frac{N}{\hat x} (1 - \beta).
\end{equation}

To apply the conservation of mass law lets find the number of cars that will arrive to the destination area as the time needed for all the cars of the first part of the flow to arrive to the destination area times the capacity of the bottle neck

\begin{equation}
 {\hat x (T_0 - \bar t - t_1)} = {\hat x (T_0 - \bar t - T_0 + \bar t + \frac{\alpha}{\alpha + \beta}  \frac{N}{\hat x})} = \frac{\alpha N}{\alpha + \beta}. 
\end{equation}

Now we can find the density of the first part of the flow

\begin{equation}
n_1 = \frac{\hat x (T_0 - \bar t - t_1)}{t_2 - t_1} = \frac{\frac{\alpha}{\alpha + \beta}  N}{\frac{\alpha }{(\alpha + \beta)}  \frac{N}{\hat x} (1 - \beta)} = \frac{\hat x}{1 - \beta}.
\end{equation}

Lets find the time taken by the second part of the flow to leave the departure area

\begin{equation}
{t_3 - t_2} = T_0 - \bar t + \frac{ \beta}{(\alpha + \beta)}  \frac{N}{\hat x} - T_0 + \bar t + \frac{\alpha \beta}{\alpha + \beta}  \frac{N}{\hat x} =  \frac{\beta }{(\alpha + \beta)}  \frac{N}{\hat x} (1 + \alpha).
\end{equation}

Adding together our two flows we should end up with a total number of cars. Using that principle allows find the density of the second part of the flow:

\begin{equation}
n_2 = \frac{N - n_1 (t_2 - t_1) }{t_3 - t_2} = \frac{N -\frac{\hat x}{(1 - \beta)}  \frac{\alpha }{(\alpha + \beta)}  \frac{N}{\hat x} (1 - \beta)}{\frac{\beta }{(\alpha + \beta)}  \frac{N}{\hat x} (1 + \alpha)} 
=\frac{\frac{\alpha + \beta - \alpha }{(\alpha + \beta)} }{\frac{\beta }{(\alpha + \beta)}  \frac{(1 + \alpha)}{\hat x} } = \frac{\hat x}{1 + \alpha}. 
\end{equation}

Now we can describe our car flow as a set of inequalities for time borders and the densities of each part

\begin{equation}
n(t) = 
\begin{cases}
    0, t < t_1 \\
    n_1, t_1 \le t < t_2 \\
    n_2, t_2 \le t < t_3 \\
    0, t_3 \geqslant t.
\end{cases} \\
\end{equation}






\section{Mathematical Formulation in Discrete Time}\label{sec:discrete}
We discretize time into small intervals of length $\Delta t$. These intervals must be sufficiently small to ensure that no vehicle traverses more than one link in a single interval. Typically, a value in the range of 5--10 seconds is suitable for $\Delta t$.
We normalize the time axis by a~time quant $\Delta t$, therefore $t \in \cT = \{1, \ldots, T\}$.




To formulate the discrete equilibrium problem, we define a~time-expanded transportation graph~\cite{boyles2023transportation}, where each node of the original transportation graph is multiplied by the number of timepoints and links are multiplied with account of minimum time required to travel through a link.
Specifically, we make a copy of each regular node (not centroid) for each timepoint $\tilde \cV = \cbr{V \setminus (I\cup J)} \times \cT$, and define the road segments $\cE_R$ and junctions link sets $\cE_J$ by 
\begin{equation}
  \cE_R = 
  \br{\cbr{(v_1, t_1), (v_2, t_2)}: (v_1, v_2) \in E_R,\, t_2 - t_1 \geq \bar t_{(v_1, v_2)}}
\end{equation}
and 
\begin{equation}
  \cE_J = 
  \br{\cbr{(v_1, t), (v_2, t + \bar t_{(v_1, v_2)})}: (v_1,v_2) \in E_J,\, t \in \{0, 1, \ldots, T - \bar t_{(v_1, v_2)}\}}.
\end{equation}
This means that vehicles can spend arbitrary time waiting in point queue on a link, but times for passing junctions are fixed. 
For each link $\cbr{(v_1, t_1), (v_2, t_2)}$ its cost is set to $t_2 - t_1$.

We connect each origin centroid $i \in I$ to each copy of each node it was connected to in graph $G$.
To embed the arrival costs in the time-expanded graph, we make time copies for each destination centroid, redefining $J$ as $J = J \times \cT$ for users that aim to arrive at different moments of time.
We also make another copies of destinations $J' = J \times \cT$ for users who actually arrived at different moments of time and connect them to all corresponding copies of junction nodes and to all copies of same centroids in $J$.
Travel costs $\bar t_e$ on links $\cbr{(j', t'), (j, t)}$, $j' \in J'$, $j \in J$ represent cost of arrival at zone $j$ in time $t'$ for users who aimed to arrive at time $t$.

We denote the time-expanded transportation graph described above as $\cG = (\cV, \cE)$.

\begin{figure}[!ht]
\centering
\includegraphics[scale=0.74]{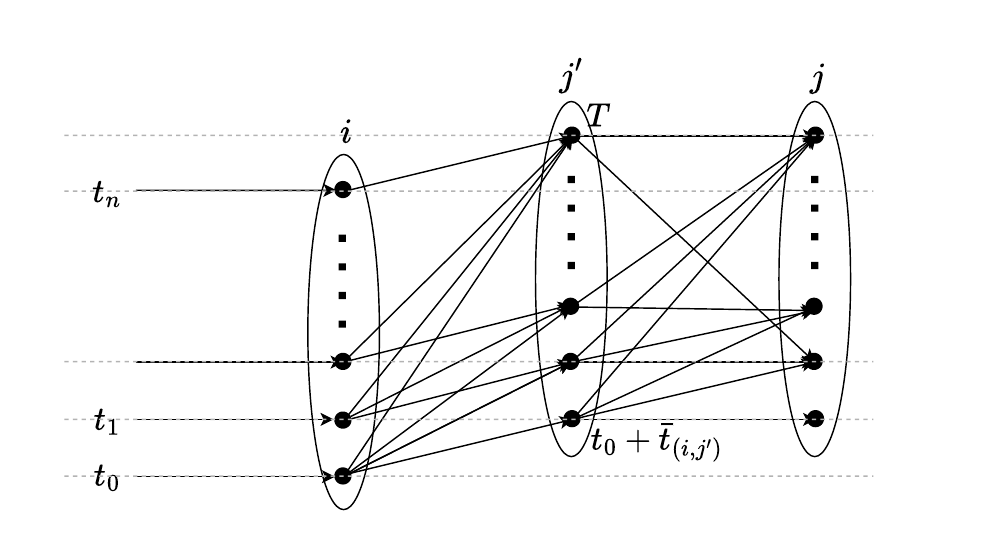}
\smallskip
\caption{Example of time-expanded graph. Node $i$ is a source node with different departure times. Node $j'$ is a destination zone witth different arrival times. Node $j$ shows aimed arrival time. Edges between $i$ and $j'$ are physical edges and have some travel costs, however edges between $j'$ and $j$ does not exists in reality and used to introduce arrival costs.}
\label{fig:time_dependent_graph}
\end{figure}

\subsection{Mathematical Formulations}
\textbf{Complementarity.}
By the first Wardrop's principle, the equilibrium time and route choice problem formulates as the problem of assigning a departure time and a path in the spatial transportation network for each user in such a way that no user can change his strategy (departure time and spatial path) to strictly decrease travel cost.
If we embed the departure time in the definition of path, we can say that in equilibrium each user is assigned to a shortest path.

For single spatial path, time spent in queue may depend on the network loading.
But in time-expanded graph, each path has fixed travel cost.
The way, in which users affect each other choices, from the time-expanded graph point of view, is that some paths may be prohibited to use, because they are already occupied  by other agents so that their capacity limit is reached.
Therefore, we need to distinguish available and fully occupied paths.



Mathematically, the equilibrium problem can be stated as a nonlinear complementarity problem:
\newcommand{\freeP}{P^{\text{avail}}}
\begin{alignat}{2}
  (T_{p'} - T_p) x_{p} &\geq 0  &\; \forall p \in P_{ij},\, \forall p' \in \freeP_{ij}(x),\, \forall i \in I,~ j\in J \label{eq:compl}
  \\
  x_p &\geq 0 &\forall p \in P,\label{eq:nonneg}
  \\
  \sum_{p \in P, p \ni e} x_p &\leq \bar x_e &\forall e \in \cE,\label{eq:caps}
  \\
  \sum_{p \in P_{ij}} x_p &= d_{ij} &\forall i \in I,~ j\in J,\label{eq:demand}
\end{alignat}
where  $P_{ij}$ is the set of all paths from origin node $i\in I$ to  destination node $j \in J$ of the $ij$-th demand in the time-expanded graph; $\freeP_{ij}(x)$ is the set of paths $p \in P_{ij}$ such that flow $x_p$ can be increased by a $\delta_p > 0$ without violating capacity constraints \eqref{eq:caps};
$P = \bigcup_{i\in I, j \in J} P_{ij}$; $x$ is the path flows vector with components $x_p,\, p\in P$; $T_p$ is the travel cost for path $p$;
The meaning of the inequalities is as follows: \eqref{eq:nonneg} ensures that flows are non-negative; \eqref{eq:compl} stands for either $x_p = 0$ or 
the travel cost for path $p$ is less or equal than the cost of the shortest available path for $ij$-th demand (thus the problem is called complementarity problem); and \eqref{eq:caps},\eqref{eq:demand} express capacity and demand constraints respectively.


\textbf{Variational Inequality.}
Let $X$ denote the set of all feasible path flows vectors, i.e. path flows vectors  satisfying \eqref{eq:nonneg}-\eqref{eq:demand}.
Then, the nonlinear complementarity problem can be reformulated as the following variational inequality problem (VI) on $x$: 
\begin{equation}
  x \in X, \quad \sum_{p \in P} T_p (x_p' - x_p) \geq 0 \; \forall x' \in X \label{eq:vi}. \tag{\ref{eq:compl}'} 
\end{equation}

\begin{lemma}\label{lem:vi_to_equil}
  Any solution to VI \eqref{eq:vi} is an equilibrium assignment \eqref{eq:compl}-\eqref{eq:demand}
\end{lemma}
\begin{proof}
  Since $x \in X$ simply denotes \eqref{eq:nonneg}-\eqref{eq:demand}, we only require to show that \eqref{eq:compl} follows from the inequality in \eqref{eq:vi}.
  Assume the opposite, that there exist $i\in I,~ j \in J$ and $p \in P_{ij}, p' \in \freeP_{ij}$ such that $x_p > 0$ and $T_{p'} < T_p$.
  Then we can simply define $x'$ by setting $x'_\pi = x_\pi$ for all 
  $\pi \in P \setminus \{p, p'\}$, $x'_{p'} = x_{p'} + \min\br{\delta_{p'}, x_{p}}$ and $x'_p = x_p - \min\br{\delta_{p'}, x_{p}}$.
  It gives us $\sum_{p \in P} T_p (x_p' - x_p) = (T_{p'} - T_p)\min\br{\delta_{p'}, x_{p}} < 0$, what contradicts \eqref{eq:vi}.
\end{proof}

\textbf{Linear Programming.}
One can note that \eqref{eq:vi} coincides with the first-order optimality criteria for a problem of minimizing a convex function $f(x)$ on a convex set $X$ if $\partial f(x) / \partial x_p = T_p$ \cite[Propositon~1.3]{bubeck2014convex}.
The set $X$ defined above is indeed convex, and the convex function $f(x) = \sum_{p \in P} x_p T_p$, which is equal to total travel cost, satisfies the equality for partial derivatives. 
Therefore, the following convex optimization problem
\begin{equation}\label{eq:opt}
\min_{x \in X} \sum_{p \in P} x_p T_p,
\end{equation}
which, in fact, is a linear programming problem, is equivalent to VI \eqref{eq:vi}.
Let us assume that $X$ is nonempty. Due to \eqref{eq:nonneg} and \eqref{eq:caps}, $X$ is bounded. This implies that problem~\eqref{eq:opt} has a solution.
Together with \cref{lem:vi_to_equil} this gives us the following theorem.
\begin{theorem}
  Let $X$ be nonempty set of path flows vectors $x$ defined by \eqref{eq:nonneg}-\eqref{eq:demand}.
  Then, the equilibrium time and route problem \eqref{eq:compl}-\eqref{eq:demand} has a solution.
  This solution is a solution to the linear programming problem~\eqref{eq:opt}.
\end{theorem}

\section{Algorithmic Approaches}\label{sec:algo}
The linear programming problem~\eqref{eq:opt} can further be classified as a minimum cost multicommodity flow (MCMF) problem~\cite{ahuja1988network}.
Depending on the network size, number of timepoints $|\cT|$ and the desired solution accuracy, different algorithmic approaches can be preferred for solving the MCMF problem.
Here we discuss several methods appropriate to the specifics of the problem related with the properties of time-expanded graph.
For comprehensive review of algorithms for multicommodity flow problems we refer to~\cite{wang2018multicommodity2}.

\subsection{Link-Based}\label{sec:link-based}
Problem~\eqref{eq:opt} is defined for the path flows vector variable.
The problem can also be reformulated for link flows variables \cite{ahuja1988network}:
\begin{align}
  &\min_{\substack{x_{ei} \geq 0,\\ e \in \cE, i \in I}} \sum_{e \in \cE} \bar t_e x_e
  \\\text{s.t. }&
  x_e = \sum_{i \in I} x_{ei} \quad \forall e \in \cE
  \\&
  x_e \leq \bar x_e \quad \forall e \in \cE \label{eq:link_caps}
  \\&\hspace{-0.5cm}
  \sum_{e \in \text{Out}(v)} x_{ei} -  \sum_{e \in \text{In}(v)} x_{ei} = 
  \begin{cases}
    0, \; \text{if } v \notin J \text{ and } v \neq i,\\
    \sum_{j \in J}d_{ij},  \;\text{if } v = i,\\
    -d_{iv},  \;\text{if }  v \in J,
  \end{cases}
  \forall v \in \cV,~ i \in I, \label{eq:kirch}
\end{align}
where $x_{ei}$ denotes flow on edge $e$ originating from origin $i\in I$; $x_e$ denotes total flow on edge $e$; $\text{In}(v) \subset \cE$ and $\text{Out}(v) \subset \cE$ denote sets of links coming in and out of node $v \in \cV$ respectively.
Equation \eqref{eq:kirch} imposes flow continuity constraint in form of Kirchhoff's current law.

In this form MCMF problem has $O(mT|I|)$ variables and $O(nT|I|)$ constraints.
Constraint \eqref{eq:kirch} can be written in a matrix form using the incidence matrix of the time-expanded graph.
This allows to apply existing linear programming solvers \cite{huangfu2018parallelizing,diamond2016cvxpy}, but for larger networks or finer time discretization this approach may not work due to the problem's size.

\subsection{Column Generation}\label{sec:column-generation}
The column generation approach~\cite[Section~17.5]{ahuja1988network} can be applied to solve the MCMF problem in the path-flows form~\eqref{eq:opt}.
Basically, it is a specific variant of simplex method, which tackle the problem of huge dimension of $P$ by operating only on a limited subset of paths from $P$ which carry nonzero flow.
Column generation simplex method at each iteration solves the shortest paths problem for each demand and then solves a linear system of equations to distribute path flows between previously used paths and newly found shortest paths.
Dual variable's values are updated to satisfy a condition that all active paths for given demand have the same travel costs (with respect to the additional costs introduced by the dual variables).
The primal variables in the linear system, also called basic variables, are nonzero path flows and nonzero slack variables for constraint~\eqref{eq:link_caps}.
The number of basic variables is the number of paths with nonzero flow plus the number of not fully loaded link, and it is equal to $|K| + |\cE|$.
Since $|\cE| = O(mT)$, the system can be prohibitively large.

However, it is likely that only a small subset of links will carry nonzero flow at each iteration: for every $t \in \cT,~ u,v \in V$ links $\cbr{(u,t), (v, t')}$ can be expected to be used only for $t' \geq t$ in a small subset of $\cT$, corresponding to limited variation of waiting time between users in the same queue.
After solving the shortest path subproblem, unused links can be removed from consideration for the linear system solving step, greatly reducing number of variables in the system.
With this optimization in mind, we can expect each step of the column generation simplex method require solving linear system on $O(m + |K|)$ variables and solving the shortest path problem on network with $O(Tn)$ nodes and $O(Tm)$ links for each demand, what can be done by $|I|$ invocations of Dijkstra's algorithm.

\subsection{Dual Subgradient Methods}\label{sec:dual-subgd}
By using Lagrangian duality to handle the capacity constraints \eqref{eq:caps} in problem \eqref{eq:opt}, one can obtain the following primal-dual problem~\cite[Section~17.4]{ahuja1988network}:
\begin{equation}
  \max_{\substack{y_e \geq 0, \\ e \in \cE}} \min_{\substack{x_p \geq 0, p \in P \\ \sum\limits_{p \in P_{ij}} x_p = d_{ij}}} 
  \br{\sum_{e \in \cE} \bar t_e \sum_{p \ni e}x_p + \sum_{e \in \cE}y_e \cbr{\sum_{p \ni e}x_p - \bar x_e}},
\end{equation}
where dual variables $y_e$ penalize usage of congested links.
As in column generation approach, one iteration of a subgradient method for solving the dual problem (by variable $y$) require finding shortest paths for all demands with respect to additional costs introduced with dual variables. 
Demands are then distributed over the shortest paths what induces the so-called all-or-nothing traffic assignment on links.
Comparing to column generation, the linear system solving step is replaced with a much simpler operation: 
the flow distribution from the previous iteration is averaged with the obtained all-or-nothing assignment.
Dual variables values $y_e$ are increased for congested links and decreased for links with available capacity (proportionally to the difference $\sum_{p \ni e}x_p - \bar x_e$ in case of classical subgradient method).
Again, only variables, corresponding to active (used) links are updated at each iteration. 

In \cite{kubentayeva2021finding} a detailed discussion of this approach is given for general MCMF problem (referred to by the alias of the stable dynamic model) with convergence analysis of adaptive primal-dual subgradient methods.

\medskip
Since practically efficient simplex method has unsatisfying worst-case complexity~\cite{klee1972good}, we cannot make informative analytical comparison of complexities of above mentioned approaches.
However, one can expect link-based approach~\ref{sec:link-based} only work for small networks, while column generation~\ref{sec:column-generation} and dual subgradient~\ref{sec:dual-subgd} methods can be applied for large networks with the latter have computationally more affordable but less effective iteration.

\section{Conclusion}
The proposed dynamic model has relatively low computational complexity, and suitable solution algorithms for various scenarios were proposed. 
It has no guarantees for the uniqueness of the equilibrium, but the system optimum solution is always an equilibrium and significant differences in travel costs for different equilibria may indicate inefficient system design.
More physical modelling of link flows than in the Beckmann model and broader application range than that of the Stable Dynamics model (e.g. peak hours) give the potential to use the proposed model as a compromise between static and more complex dynamic approaches e.g. in combined travel-demand models.

\begin{credits}
\subsubsection{\ackname} 
The research is supported by the Ministry of Science and Higher Education of the Russian Federation (Goszadaniye) No. 075-03-2024-117, project No. FSMG-2024-0025.
\end{credits}

\bibliographystyle{splncs04}
\bibliography{references}
\end{document}